\documentclass[12pt, reqno]{amsart}
\UseRawInputEncoding
\usepackage{amsmath, amsthm, amscd, amsfonts, amssymb, graphicx, color}
\usepackage[bookmarksnumbered, colorlinks, plainpages]{hyperref}
\input{mathrsfs.sty}

\hypersetup{colorlinks=true,linkcolor=red, anchorcolor=green,
citecolor=cyan, urlcolor=red, filecolor=magenta, pdftoolbar=true}

\newtheorem{theorem}{Theorem}[section]
\newtheorem{lemma}[theorem]{Lemma}

\newtheorem{corollary}[theorem]{Corollary}
\theoremstyle{definition}
\newtheorem{definition}[theorem]{Definition}

\theoremstyle{remark}
\newtheorem{remark}[theorem]{Remark}
\numberwithin{equation}{section}

\begin{document}

\setcounter{page}{1}

\title[Numerical radius in Hilbert $C^*$-modules]
{Numerical radius in Hilbert $C^*$-modules}

\author[A. Zamani]{Ali Zamani}

\address{School of Mathematics and Computer Sciences, Damghan University, Damghan, P.~O.~BOX 36715-364, Iran}
\email{zamani.ali85@yahoo.com}


\begin{abstract}
Utilizing the linking algebra of a Hilbert $C^*$-module $\big(\mathscr{V}, {\|\!\cdot\!\|}\big)$,
we introduce $\Omega(x)$ as a definition of numerical radius for an element $x\in\mathscr{V}$ and then
show that $\Omega(\cdot)$ is a norm on $\mathscr{V}$ such that
$\frac{1}{2}{\|x\|} \leq \Omega(x) \leq {\|x\|}$.
In addition, we obtain an equivalent condition
for $\Omega(x) = \frac{1}{2}{\|x\|}$.
Moreover, we present a refinement of the triangle inequality for the norm $\Omega(\cdot)$.
Some other related results are also discussed.
\end{abstract} \maketitle
\section{Introduction and preliminaries}
The notion of Hilbert $C^*$-module is a natural generalization of that of Hilbert space
arising under replacement of the field of scalars $\mathbb{C}$ by a $C^*$-algebra.
This concept plays a significant role in the theory of operator algebras,
quantum groups, noncommutative geometry and $K$-theory; see \cite{L, M.T}.

Let us give that some necessary background and set up our notation.
An element $a$ in a $C^*$-algebra $\mathscr{A}$ is called positive (we write $0 \leq a$)
if $a = b^*b$ for some $b\in \mathscr{A}$.
For an element $a$ of $\mathscr{A}$, we denote by
\begin{align*}
\,{\rm Re}\,a = \frac{1}{2}(a + a^*), \quad \,{\rm Im}\,a = \frac{1}{2i}(a - a^*)
\end{align*}
the real and the imaginary part of $a$.
By $\mathscr{A}'$ we denote the dual space of $\mathscr{A}$. A positive linear functional of $\mathscr{A}$ is a map
$\varphi\in \mathscr{A}'$ such that $0\leq\varphi(a)$ whenever $0\leq a$. The set of all states of $\mathscr{A}$, that
is, the set of all positive linear functionals of $\mathscr{A}$ of norm $1$, is denoted by $\mathcal{S}(\mathscr{A})$.
An inner product module over $\mathscr{A}$ is a (left)
$\mathscr{A}$-module $\mathscr{V}$ equipped with an $\mathscr{A}$-valued
inner product $\langle\cdot, \cdot\rangle$, which is $\mathbb{C}$-linear and
$\mathscr{A}$-linear in the first variable and has the properties $\langle x, y
\rangle^*=\langle y, x\rangle$ as well as $0\leq\langle x, x\rangle$ with equality
if and only if $x = 0$. The $\mathscr{A}$-module $\mathscr{V}$ is called a Hilbert $\mathscr{A}$-module
if it is complete with respect to the norm ${\|x\|} = {\|\langle x, x\rangle\|}^{\frac{1}{2}}$.
In a Hilbert $\mathscr{A}$-module $\mathscr{V}$ we have the following version of the Cauchy–-Schwarz inequality:
\begin{align}\label{C.S.I}
\langle y, x\rangle\langle x, y\rangle \leq\|x\|^2\langle y, y\rangle, \qquad (x, y\in\mathscr{V}).
\end{align}
Every $C^*$-algebra $\mathscr{A}$ can be regarded as a Hilbert $C^*$-module over itself where the inner product is
defined by $\langle a, b\rangle = a^*b$.
Let $\mathscr{V}$ and $\mathscr{W}$ be two Hilbert $\mathscr{A}$-modules.
A mapping $T\,:\mathscr{V}\longrightarrow \mathscr{W}$ is called adjointable if there exists a mapping
$S\,:\mathscr{W}\longrightarrow \mathscr{V}$ such that $\langle Tx, y\rangle=\langle x, Sy\rangle$
for all $x\in \mathscr{V}, y\in \mathscr{W}$. The unique mapping $S$ is denoted by $T^*$ and
is called the adjoint operator of $T$. The space $\mathbb{B}(\mathscr{V},\mathscr{W})$
of all adjointable maps between Hilbert $\mathscr{A}$-modules $\mathscr{V}$ and $\mathscr{W}$
is a Banach space, while $\mathbb{B}(\mathscr{V}) :=\mathbb{B}(\mathscr{V},\mathscr{V})$ is a $C^*$-algebra.
By $\mathbb{K}(\mathscr{V},\mathscr{W})$ we denote the closed linear subspace of $\mathbb{B}(\mathscr{V},\mathscr{W})$
spanned by $\big\{\theta_{x, y}: \, x\in \mathscr{W}, y\in \mathscr{V}\big\}$,
where $\theta_{x, y}$ is defined by $\theta_{x, y}(z) = x\langle y, z\rangle$.
Elements of $\mathbb{K}(\mathscr{V},\mathscr{W})$ are often referred to as ``compact'' operators.
We write $\mathbb{K}(\mathscr{V})$ for $\mathbb{K}(\mathscr{V},\mathscr{V})$.
Given a Hilbert $\mathscr{A}$-module $\mathscr{V}$, the linking algebra $\mathbb{L}(\mathscr{V})$
is defined as the matrix algebra of the form
\begin{align*}
\mathbb{L}(\mathscr{V})=\begin{bmatrix}
\mathbb{K}(\mathscr{A}) & \mathbb{K}(\mathscr{V}, \mathscr{A}) \\
\mathbb{K}(\mathscr{A}, \mathscr{V}) & \mathbb{K}(\mathscr{V})
\end{bmatrix}.
\end{align*}
Then $\mathbb{L}(\mathscr{V})$ has a canonical embedding as a closed subalgebra of the adjointable
operators on the Hilbert $\mathscr{A}$-module $\mathscr{A}\oplus \mathscr{V}$ via
\begin{align*}
\begin{bmatrix}
X & Y \\
Z & W
\end{bmatrix}\begin{bmatrix}
a\\
x
\end{bmatrix}=\begin{bmatrix}
Xa + Yx \\
Za + Wx
\end{bmatrix}
\end{align*}
which makes $\mathbb{L}(\mathscr{V})$ a $C^*$-algebra (cf. \cite{R.W}, Lemma~2.32 and Corollary~3.21).
Each $x\in \mathscr{V}$ induces the maps ${r}_{x}\in\mathbb{B}(\mathscr{A}, \mathscr{V})$ and ${l}_{x}\in\mathbb{B}(\mathscr{V}, \mathscr{A})$
given by ${r}_{x}(a) = xa$ and ${l}_{x}(y)=\langle x, y\rangle$, respectively, such that ${r}^*_{x} = {l}_{x}$.
The map $x\mapsto {r}_{x}$ is an isometric linear isomorphism of $\mathscr{V}$
to $\mathbb{K}(\mathscr{A}, \mathscr{V})$ and $x\mapsto {l}_{x}$ is an isometric conjugate linear isomorphism of
$\mathscr{V}$ to $\mathbb{K}(\mathscr{V}, \mathscr{A})$. Further, every $a\in \mathscr{A}$
induces the map ${T}_{a}\in\mathbb{K}(\mathscr{A})$ given by ${T}_{a}(b) = ab$. The
map $a\mapsto {T}_{a}$ defines an isomorphism between $C^*$-algebras $\mathscr{A}$ and $\mathbb{K}(\mathscr{A})$.
Therefore, we may write
\begin{align*}
\mathbb{L}(\mathscr{V}) = \left\{\begin{bmatrix}
{T}_{a} & {l}_{y} \\
{r}_{x} & T
\end{bmatrix} : a\in\mathscr{A},\, x, y\in \mathscr{V},\, T\in\mathbb{K}(\mathscr{V})\right\},
\end{align*}
and identify the $C^*$-subalgebras of compact operators with the corresponding corners in the linking algebra:
$\mathbb{K}(\mathscr{A}) = \mathbb{K}(\mathscr{A}\oplus 0) \subseteq \mathbb{K}(\mathscr{A}\oplus \mathscr{V}) = \mathbb{L}(\mathscr{V})$ and
$\mathbb{K}(\mathscr{V}) = \mathbb{K}(0\oplus \mathscr{V})\subseteq \mathbb{K}(\mathscr{A}\oplus \mathscr{V}) = \mathbb{L}(\mathscr{V})$.
We refer the reader to \cite{L, M.T} for more information on Hilbert $C^*$-modules and linking algebras.

Now, let $\mathbb{B}(\mathscr{H})$ denote the $C^*$-algebra of all bounded linear
operators on a complex Hilbert space $\mathscr{H}$ with inner product $[\cdot, \cdot]$.
The numerical range of an element $A\in\mathbb{B}(\mathscr{H})$ is defined
\begin{align*}
W(A):= \left\{[A\xi, \xi]: \, \xi \in\mathscr{H}, {\|\xi\|} = 1\right\}.
\end{align*}
It is known that $W(A)$ is a nonempty bounded convex subset of $\mathbb{C}$ (not necessarily
closed). This concept is useful in studying linear operators and have attracted the attention of many authors
in the last few decades (e.g., see \cite{G.R}, and references therein).
The numerical radius of $A$ is given by
\begin{align*}
w(A)= \sup\left\{|[A\xi, \xi]|: \, \xi \in\mathscr{H}, {\|\xi\|} = 1\right\}.
\end{align*}
It is known that $w(\cdot)$ is a norm on $\mathbb{B}(\mathscr{H})$ and satisfies
\begin{align*}
\frac{1}{2}{\|A\|} \leq w(A) \leq {\|A\|}
\end{align*}
for each $A\in\mathbb{B}(\mathscr{H})$. Some generalizations of the numerical radius
$A\in\mathbb{B}(\mathscr{H})$ can be found in \cite{A.K.2, Z.W}.

In the next section, we first utilize the linking algebra $\mathbb{L}(\mathscr{V})$ of a Hilbert $\mathscr{A}$-module $\mathscr{V}$
to introduce $\Phi(x)$ as a definition of numerical range for an arbitrary element $x\in\mathscr{V}$.
We then use this set to define numerical radius of $x$ and denote it by $\Omega(x)$.
In particular, we show that $\Omega(\cdot)$ is a norm on $\mathscr{V}$, which is equivalent
to the norm ${\|\!\cdot\!\|}$ and the following inequalities hold for every $x\in \mathscr{V}$:
\begin{align}\label{I.1.1}
\frac{1}{2}{\|x\|} \leq \Omega(x) \leq {\|x\|}.
\end{align}
We also establish an inequality that refines the first inequality in (\ref{I.1.1}).
In addition, we prove that $\Omega(x) = \frac{1}{2}{\|x\|}$ if and only if
${\|x\|} = {\left\|\begin{bmatrix}
0 & \overline{\lambda}{l}_{x}
\\ \lambda {r}_{x} & 0
\end{bmatrix}\right\|}$ for all complex unit $\lambda$.
Furthermore, for $x\in \mathscr{V}$ and $a\in \mathscr{A}$ we prove that
\begin{align*}
\Omega(xa \pm xa^*) \leq 2{\|a \pm a^*\|} \Omega(x).
\end{align*}
We finally present a refinement of the triangle inequality for the norm $\Omega(\cdot)$.
\section{Main results}
We start our work with the following definition.
\begin{definition}\label{D.2.00}
Let $\mathscr{V}$ be a Hilbert $\mathscr{A}$-module
and let $\mathbb{L}(\mathscr{V})$ be the linking algebra of $\mathscr{V}$.
The numerical range of $x\in \mathscr{V}$ is defined as the set
\begin{align*}
\Phi(x):= \left\{\varphi\left(\begin{bmatrix}
0 & 0
\\ {r}_{x} & 0
\end{bmatrix}\right): \, \varphi \in \mathcal{S}\big(\mathbb{L}(\mathscr{V})\big)\right\}.
\end{align*}
\end{definition}
Next, we present some properties of the numerical range in Hilbert $C^*$-modules.
\begin{theorem}\label{T.2.01}
Let $x$ and $y$ be elements of a Hilbert $\mathscr{A}$-module $\mathscr{V}$ and let $\alpha \in \mathbb{C}$.
Then
\begin{itemize}
\item[(i)] $\Phi(\alpha x) = \alpha \Phi(x)$ (homogeneous).
\item[(ii)] $\Phi(x+y)\subseteq \Phi(x) + \Phi(y)$ (subadditive).
\item[(iii)] $\Phi(x)$ is a nonempty compact convex subset of $\mathbb{C}$.
\end{itemize}
\end{theorem}
\begin{proof}
Let $\mathbb{L}(\mathscr{V})$ be the linking algebra of $\mathscr{V}$.
For every $a\in\mathscr{A}$, we have
\begin{align*}
{r}_{\alpha x}(a) = (\alpha x)a = \alpha (xa) = \big(\alpha{r}_{x}\big)(a)
\end{align*}
and
\begin{align*}
{r}_{x+y}(a) = (x+y)a = xa + ya = \big({r}_{x} + {r}_{y}\big)(a).
\end{align*}
Hence ${r}_{\alpha x} = \alpha{r}_{x}$ and ${r}_{x+y} = {r}_{x} + {r}_{y}$.
Thus (i) and (ii) follow easily from the definition.

We now prove (iii). Since the existence of states on $\mathbb{L}(\mathscr{V})$
is guaranteed by the Hahn--Banach theorem,
we have $\Phi(x) \neq \emptyset$. The convexity of $\Phi(x)$ is an easy consequence
of the fact that a convex combination of two states is also a state. As for the compactness,
note that the set $\mathcal{S}\big(\mathbb{L}(\mathscr{V})\big)$ is a weak*-closed subset of the unit ball
$\left\{\varphi\in \mathbb{L}^{'}(\mathscr{V}):\, \|\varphi\|\leq 1\right\}$ of $\mathbb{L}^{'}(\mathscr{V})$.
Since, by the Banach--Alaoglu theorem, the latter is weak*-compact, the same is true for $\mathcal{S}\big(\mathbb{L}(\mathscr{V})\big)$.
Hence $\Phi(x)$, the image of the weak*-continuous mapping $\varphi \mapsto \varphi\left(\begin{bmatrix}
0 & 0
\\ {r}_{x} & 0
\end{bmatrix}\right)$ for $\varphi\in\mathcal{S}\big(\mathbb{L}(\mathscr{V})\big)$, is compact in $\mathbb{C}$.
\end{proof}
\begin{remark}
It is known that the set of all states of a unital $C^*$-algebra $\mathscr{A} \subseteq \mathbb{B}(\mathscr{H})$
is a weak*-closed convex hull of the set of all vector states of $\mathscr{A}$, i.e., the
states of $\mathscr{A}$ of the form $A \rightarrow [A\xi, \xi]$ for some unit vector $\xi$ in $\mathscr{H}$.
Also, for the Hilbert module $\mathscr{V}=\mathbb{B}(\mathscr{H})$ over the $C^*$-algebra $\mathbb{B}(\mathscr{H})$ is well known to be valid
$\mathbb{K}(\mathbb{B}(\mathscr{H}))= \mathbb{K}(\mathscr{V}, \mathbb{B}(\mathscr{H})) = \mathbb{K}(\mathbb{B}(\mathscr{H}), \mathscr{V})
= \mathbb{K}(\mathscr{V}) = \mathbb{B}(\mathscr{H})$ (see \cite[Remark~1.13]{B.G}), so all corners in the linking algebra $\mathbb{L}(\mathscr{V})$
are equal to $\mathbb{B}(\mathscr{H})$. Hence, for $A\in \mathbb{B}(\mathscr{H})$, we have $\Phi(A) = \overline{W(A)}$.
\end{remark}
Now, we are in a position to introduce numerical radius for elements of a Hilbert $C^*$-module.
Some other related topics can be found in \cite{A.M, C.C.S, M.A.O, Raji.1, Raji.2, T.Z.P.A}.
\begin{definition}\label{D.2.02}
Let $\mathscr{V}$ be a Hilbert $\mathscr{A}$-module
and let $\mathbb{L}(\mathscr{V})$ be the linking algebra of $\mathscr{V}$.
The numerical radius of an element $x\in \mathscr{V}$ is defined as
\begin{align*}
\Omega(x):= \sup\left\{\left|\varphi\left(\begin{bmatrix}
0 & 0
\\ {r}_{x} & 0
\end{bmatrix}\right)\right|: \, \varphi \in \mathcal{S}\big(\mathbb{L}(\mathscr{V})\big)\right\}.
\end{align*}
\end{definition}
In the following theorem, we prove that $\Omega(\cdot)$ is a norm on Hilbert $C^*$-module $\mathscr{V}$,
which is equivalent to the norm ${\|\!\cdot\!\|}$.
\begin{theorem}\label{T.2.2}
Let $\mathscr{V}$ be a Hilbert $\mathscr{A}$-module.
Then $\Omega(\cdot)$ is a norm on $\mathscr{V}$ and the following inequalities hold for every $x\in \mathscr{V}$:
\begin{align*}
\frac{1}{2}{\|x\|} \leq \Omega(x) \leq {\|x\|}.
\end{align*}
\end{theorem}
\begin{proof}
Let $\mathbb{L}(\mathscr{V})$ be the linking algebra of $\mathscr{V}$.
Let $x\in \mathscr{V}$. Clearly, $\Omega(x)\geq 0$.
Let us now suppose $\Omega(x) = 0$. Then, by Definition \ref{D.2.02},
$\begin{bmatrix}
0 & 0
\\ {r}_{x} & 0
\end{bmatrix} = 0$.
Since ${\left\|\begin{bmatrix}
0 & 0
\\ {r}_{x} & 0
\end{bmatrix}\right\|}
= {\|x\|}$,
we get ${\|x\|}=0$ and therefore, $x = 0$.
Further, by Theorem \ref{T.2.01} (i)-(ii), for $y, z \in \mathscr{V}$ and $\alpha \in \mathbb{C}$ we have
$\Omega(\alpha y) = |\alpha|\Omega(y)$ and
$\Omega(y + z)\leq \Omega(y) + \Omega(z)$. Thus $\Omega(\cdot)$ is a norm on $\mathscr{V}$.

On the other hands, for every $\varphi\in\mathcal{S}\big(\mathbb{L}(\mathscr{V})\big)$, we have
\begin{align*}
\left|\varphi\left(\begin{bmatrix}
0 & 0
\\ {r}_{x} & 0
\end{bmatrix}\right)\right| \leq
{\left\|\begin{bmatrix}
0 & 0
\\ {r}_{x} & 0
\end{bmatrix}\right\|}
= {\|x\|}.
\end{align*}
So, by taking the supremum over $\varphi\in\mathcal{S}\big(\mathbb{L}(\mathscr{V})\big)$ in the above inequality, we deduce that
\begin{align}\label{T.2.2.I.1}
\Omega(x) \leq {\|x\|}.
\end{align}
Now let
$\begin{bmatrix}
0 & 0
\\ {r}_{x} & 0
\end{bmatrix} =\,{\rm Re}\left(\begin{bmatrix}
0 & 0
\\ {r}_{x} & 0
\end{bmatrix}\right) + i\,{\rm Im}\left(\begin{bmatrix}
0 & 0
\\ {r}_{x} & 0
\end{bmatrix}\right)$ be the Cartesian decomposition of $\begin{bmatrix}
0 & 0
\\ {r}_{x} & 0
\end{bmatrix}$.
By \cite[Theorem~3.3.6]{Mu}, there exist $\varphi_1, \varphi_2\in \mathcal{S}\big(\mathbb{L}(\mathscr{V})\big)$ such that
\begin{align}\label{T.2.2.I.1.1}
\left|\varphi_1\left(\,{\rm Re}\left(\begin{bmatrix}
0 & 0
\\ {r}_{x} & 0
\end{bmatrix}\right)\right)\right| = {\left\|\,{\rm Re}\left(\begin{bmatrix}
0 & 0
\\ {r}_{x} & 0
\end{bmatrix}\right)\right\|}
\end{align}
and
\begin{align}\label{T.2.2.I.1.2}
\left|\varphi_2\left(\,{\rm Im}\left(\begin{bmatrix}
0 & 0
\\ {r}_{x} & 0
\end{bmatrix}\right)\right)\right| = {\left\|\,{\rm Im}\left(\begin{bmatrix}
0 & 0
\\ {r}_{x} & 0
\end{bmatrix}\right)\right\|}.
\end{align}
Therefore, by \eqref{T.2.2.I.1.1} and \eqref{T.2.2.I.1.2}, we have
\begin{align*}
\frac{1}{2}{\|x\|} &= \frac{1}{2}{\left\|\begin{bmatrix}
0 & 0
\\ {r}_{x} & 0
\end{bmatrix}\right\|}
\\& \leq \frac{1}{2}{\left\|\,{\rm Re}\left(\begin{bmatrix}
0 & 0
\\ {r}_{x} & 0
\end{bmatrix}\right)\right\|}
+
\frac{1}{2}{\left\|\,{\rm Im}\left(\begin{bmatrix}
0 & 0
\\ {r}_{x} & 0
\end{bmatrix}\right)\right\|}
\\& = \frac{1}{2}\left|\varphi_1\left(\,{\rm Re}\left(\begin{bmatrix}
0 & 0
\\ {r}_{x} & 0
\end{bmatrix}\right)\right)\right|
+ \frac{1}{2}\left|\varphi_2\left(\,{\rm Im}\left(\begin{bmatrix}
0 & 0
\\ {r}_{x} & 0
\end{bmatrix}\right)\right)\right|
\\& = \frac{1}{4}\left|\varphi_1\left(\begin{bmatrix}
0 & 0
\\ {r}_{x} & 0
\end{bmatrix}\right) + \overline{\varphi_1}\left(\begin{bmatrix}
0 & 0
\\ {r}_{x} & 0
\end{bmatrix}\right)\right|
+ \frac{1}{4}\left|\varphi_2\left(\begin{bmatrix}
0 & 0
\\ {r}_{x} & 0
\end{bmatrix}\right) - \overline{\varphi_2}\left(\begin{bmatrix}
0 & 0
\\ {r}_{x} & 0
\end{bmatrix}\right)\right|
\\& \leq \frac{1}{2}\left|\varphi_1\left(\begin{bmatrix}
0 & 0
\\ {r}_{x} & 0
\end{bmatrix}\right)\right|
+ \frac{1}{2}\left|\varphi_2\left(\begin{bmatrix}
0 & 0
\\ {r}_{x} & 0
\end{bmatrix}\right)\right| \leq \frac{1}{2}\Omega(x) + \frac{1}{2}\Omega(x) = \Omega(x),
\end{align*}
whence
\begin{align}\label{T.2.2.I.3}
\frac{1}{2}{\|x\|} \leq \Omega(x).
\end{align}
From (\ref{T.2.2.I.1}) and (\ref{T.2.2.I.3}), we deduce the desired result.
\end{proof}
For $A\in\mathbb{B}(\mathscr{H})$, we note that (see \cite{Y})
$w(A) = \displaystyle{\sup_{\lambda \in \mathbb{T}}}{\big\|{\rm Re}(\lambda A)\big\|}$.
Here, as usual, $\mathbb{T}$ is the unit circle of the complex plane $\mathbb{C}$.
This motivates the following result.
\begin{theorem}\label{T.2.03}
Let $\mathscr{V}$ be a Hilbert $\mathscr{A}$-module
and let $\mathbb{L}(\mathscr{V})$ be the linking algebra of $\mathscr{V}$.
Then
\begin{align*}
\Omega(x) = \frac{1}{2}\displaystyle{\sup_{\lambda \in \mathbb{T}}}{\left\|\begin{bmatrix}
0 & \overline{\lambda}{l}_{x}
\\ \lambda{r}_{x} & 0
\end{bmatrix}\right\|},
\end{align*}
for every $x\in \mathscr{V}$.
\end{theorem}
\begin{proof}
Let $x\in \mathscr{V}$. First, we show that
\begin{align}\label{T.2.0301}
\displaystyle{\sup_{\lambda \in \mathbb{T}}}\left|\,{\rm Re}\left(\lambda\varphi\left(\begin{bmatrix}
0 & 0
\\ {r}_{x} & 0
\end{bmatrix}\right)\right)\right| = \left|\varphi\left(\begin{bmatrix}
0 & 0
\\ {r}_{x} & 0
\end{bmatrix}\right)\right|
\end{align}
for every $\varphi \in \mathcal{S}\big(\mathbb{L}(\mathscr{V})\big)$.

Let $\varphi \in \mathcal{S}\big(\mathbb{L}(\mathscr{V})\big)$. We may assume that $\varphi\left(\begin{bmatrix}
0 & 0
\\ {r}_{x} & 0
\end{bmatrix}\right) \neq 0$, otherwise (\ref{T.2.0301}) trivially holds.
Put
\begin{align*}
\lambda_0 = \frac{\overline{\varphi}\left(\begin{bmatrix}
0 & 0
\\ {r}_{x} & 0
\end{bmatrix}\right)}{\left|\varphi\left(\begin{bmatrix}
0 & 0
\\ {r}_{x} & 0
\end{bmatrix}\right)\right|}.
\end{align*}
Then we have
\begin{align*}
\left|\varphi\left(\begin{bmatrix}
0 & 0
\\ {r}_{x} & 0
\end{bmatrix}\right)\right| &= \left|\,{\rm Re}\left(\lambda_0\varphi\left(\begin{bmatrix}
0 & 0
\\ {r}_{x} & 0
\end{bmatrix}\right)\right)\right|
\\&
\leq \sup_{\lambda \in \mathbb{T}}\left|\,{\rm Re}\left(\lambda\varphi\left(\begin{bmatrix}
0 & 0
\\ {r}_{x} & 0
\end{bmatrix}\right)\right)\right|
\\& \leq \sup_{\lambda \in \mathbb{T}}\left|\lambda\varphi\left(\begin{bmatrix}
0 & 0
\\ {r}_{x} & 0
\end{bmatrix}\right)\right| = \left|\varphi\left(\begin{bmatrix}
0 & 0
\\ {r}_{x} & 0
\end{bmatrix}\right)\right|,
\end{align*}
and hence (\ref{T.2.0301}) holds.

Now, since $\begin{bmatrix}
0 & \overline{\lambda}{l}_{x}
\\ \lambda{r}_{x} & 0
\end{bmatrix}$ is self adjoint for any $\lambda \in \mathbb{T}$, by \cite[Theorem~3.3.6]{Mu},
we obtain
\begin{align}\label{T.2.0302}
\left\|\begin{bmatrix}
0 & \overline{\lambda}{l}_{x}
\\ \lambda{r}_{x} & 0
\end{bmatrix}\right\|= \displaystyle{\sup_{\varphi \in \mathcal{S}(\mathbb{L}(\mathscr{V}))}}{\left|\varphi\left(\begin{bmatrix}
0 & \overline{\lambda}{l}_{x}
\\ \lambda{r}_{x} & 0
\end{bmatrix}\right)\right|}.
\end{align}
Therefore,
\begin{align*}
\displaystyle{\sup_{\lambda \in \mathbb{T}}}{\left\|\begin{bmatrix}
0 & \overline{\lambda}{l}_{x}
\\ \lambda{r}_{x} & 0
\end{bmatrix}\right\|} &\stackrel{\eqref{T.2.0302}}{=} \displaystyle{\sup_{\lambda \in \mathbb{T}}}{\displaystyle{\sup_{\varphi \in \mathcal{S}(\mathbb{L}(\mathscr{V}))}}{\left|\varphi\left(\begin{bmatrix}
0 & \overline{\lambda}{l}_{x}
\\ \lambda{r}_{x} & 0
\end{bmatrix}\right)\right|}}
\\& = 2\,\displaystyle{\sup_{\lambda \in \mathbb{T}}}{\displaystyle{\sup_{\varphi \in \mathcal{S}(\mathbb{L}(\mathscr{V}))}}{\left|\varphi\left(\,{\rm Re}\left(\lambda\begin{bmatrix}
0 & 0
\\ {r}_{x} & 0
\end{bmatrix}\right)\right)\right|}}
\\& = 2\,\displaystyle{\sup_{\lambda \in \mathbb{T}}}{\displaystyle{\sup_{\varphi \in \mathcal{S}(\mathbb{L}(\mathscr{V}))}}{\left|\,{\rm Re}\left(\lambda\varphi\left(\begin{bmatrix}
0 & 0
\\ {r}_{x} & 0
\end{bmatrix}\right)\right)\right|}}
\\& = 2{\displaystyle{\sup_{\varphi \in \mathcal{S}(\mathbb{L}(\mathscr{V}))}}\displaystyle{\sup_{\lambda \in \mathbb{T}}}{\left|\,{\rm Re}\left(\lambda\varphi\left(\begin{bmatrix}
0 & 0
\\ {r}_{x} & 0
\end{bmatrix}\right)\right)\right|}}
\\& \stackrel{\eqref{T.2.0301}}{=} 2{\displaystyle{\sup_{\varphi \in \mathcal{S}(\mathbb{L}(\mathscr{V}))}}}\left|\varphi\left(\begin{bmatrix}
0 & 0
\\ {r}_{x} & 0
\end{bmatrix}\right)\right| = 2\Omega(x).
\end{align*}
Thus
\begin{align*}
\frac{1}{2}\displaystyle{\sup_{\lambda \in \mathbb{T}}}{\left\|\begin{bmatrix}
0 & \overline{\lambda}{l}_{x}
\\ \lambda{r}_{x} & 0
\end{bmatrix}\right\|} = \Omega(x).
\end{align*}
\end{proof}
We can obtain a refinement of inequality (\ref{T.2.2.I.3}) as follows.
\begin{theorem}\label{T.2.3}
Let $\mathscr{V}$ be a Hilbert $\mathscr{A}$-module
and let $\mathbb{L}(\mathscr{V})$ be the linking algebra of $\mathscr{V}$.
For $x\in \mathscr{V}$ the following inequality holds:
\begin{align*}
\frac{1}{8}\Big(4{\|x\|} + 2|\Gamma - \Gamma'| + \Delta + \Delta'\Big)\leq \Omega(x),
\end{align*}
where $\Gamma = \max\left\{{\|x\|}, {\left\|\begin{bmatrix}
0 & {l}_{x}
\\ {r}_{x} & 0
\end{bmatrix}\right\|}\right\}$,
$\Gamma' = \max\left\{{\|x\|}, {\left\|\begin{bmatrix}
0 & -{l}_{x}
\\ {r}_{x} & 0
\end{bmatrix}\right\|}\right\}$,
$\Delta = \left|{\|x\|} -{\left\|\begin{bmatrix}
0 & {l}_{x}
\\ {r}_{x} & 0
\end{bmatrix}\right\|}\right|$
and
$\Delta' = \left|{\|x\|} -{\left\|\begin{bmatrix}
0 & -{l}_{x}
\\ {r}_{x} & 0
\end{bmatrix}\right\|}\right|$.
\end{theorem}
\begin{proof}
Since
$\Omega(x)= \frac{1}{2}\displaystyle{\sup_{\lambda \in \mathbb{T}}}{\left\|\begin{bmatrix}
0 & \overline{\lambda}{l}_{x}
\\ \lambda{r}_{x} & 0
\end{bmatrix}\right\|}$, by taking $\lambda = 1$ and $\lambda = i$, we have
\begin{align}\label{T.2.2.I.2}
\Omega(x) \geq \frac{1}{2}{\left\|\begin{bmatrix}
0 & {l}_{x}
\\ {r}_{x} & 0
\end{bmatrix}\right\|}
\qquad \mbox{and} \qquad
\Omega(x) \geq \frac{1}{2}{\left\|\begin{bmatrix}
0 & -{l}_{x}
\\ {r}_{x} & 0
\end{bmatrix}\right\|}.
\end{align}
So, by \eqref{T.2.2.I.3} and \eqref{T.2.2.I.2} we have $\Omega(x) \geq \frac{1}{2}\max\{\Gamma, \Gamma'\}$.
Therefore,
\begin{align*}
\Omega(x) & \geq \frac{\Gamma + \Gamma'}{4} + \frac{|\Gamma - \Gamma'|}{4}
\\& = \frac{1}{4}\left(\frac{1}{2}\left({\|x\|}
+ {\left\|\begin{bmatrix}
0 & {l}_{x}
\\ {r}_{x} & 0
\end{bmatrix}\right\|}\right) + \frac{1}{2}\Delta\right)
\\& \qquad \qquad \qquad + \frac{1}{4}\left(\frac{1}{2}\left({\|x\|}
+ {\left\|\begin{bmatrix}
0 & -{l}_{x}
\\ {r}_{x} & 0
\end{bmatrix}\right\|}\right) + \frac{1}{2}\Delta'\right) + \frac{|\Gamma - \Gamma'|}{4}
\\& = \frac{1}{8}\left({\left\|\begin{bmatrix}
0 & {l}_{x}
\\ {r}_{x} & 0
\end{bmatrix}\right\|} + {\left\|\begin{bmatrix}
0 & -{l}_{x}
\\ {r}_{x} & 0
\end{bmatrix}\right\|}\right)
+ \frac{1}{4}{\|x\|} + \frac{\Delta + \Delta'}{8} + \frac{|\Gamma - \Gamma'|}{4}
\\& \geq \frac{1}{8}{\left\|\begin{bmatrix}
0 & {l}_{x}
\\ {r}_{x} & 0
\end{bmatrix} + \begin{bmatrix}
0 & -{l}_{x}
\\ {r}_{x} & 0
\end{bmatrix}\right\|}
+ \frac{1}{4}{\|x\|} + \frac{\Delta + \Delta'}{8} + \frac{|\Gamma - \Gamma'|}{4}
\\& = \frac{1}{4}{\left\|\begin{bmatrix}
0 & 0
\\ {r}_{x} & 0
\end{bmatrix}\right\|}
+ \frac{1}{4}{\|x\|} + \frac{\Delta + \Delta'}{8} + \frac{|\Gamma - \Gamma'|}{4}
\\& = \frac{1}{4}{\|x\|} + \frac{1}{4}{\|x\|} + \frac{\Delta + \Delta'}{8} + \frac{|\Gamma - \Gamma'|}{4}
\\& = \frac{1}{2}{\|x\|} + \frac{\Delta + \Delta'}{8} + \frac{|\Gamma - \Gamma'|}{4}.
\end{align*}
Thus
\begin{align*}
\frac{1}{2}{\|x\|} + \frac{\Delta + \Delta'}{8} + \frac{|\Gamma - \Gamma'|}{4} \leq \Omega(x).
\end{align*}
\end{proof}
In the following result, we state a necessary and sufficient condition
for the equality case in the inequality (\ref{T.2.2.I.3}).
\begin{corollary}\label{C.2.4}
Let $\mathscr{V}$ be a Hilbert $\mathscr{A}$-module
and let $\mathbb{L}(\mathscr{V})$ be the linking algebra of $\mathscr{V}$.
Let $x\in \mathscr{V}$. Then $\Omega(x) = \frac{1}{2}{\|x\|}$ if and only if
${\|x\|} = {\left\|\begin{bmatrix}
0 & \overline{\lambda}{l}_{x}
\\ \lambda{r}_{x} & 0
\end{bmatrix}\right\|}$ for all $\lambda \in \mathbb{T}$.
\end{corollary}
\begin{proof}
Let us first suppose that $\Omega(x) = \frac{1}{2}{\|x\|}$.
For every $\lambda \in \mathbb{T}$ then we have $\Omega(\lambda x) = \frac{1}{2}{\|\lambda x\|}$.
Therefore, by Theorem \ref{T.2.3}, we obtain
\begin{align*}
\Delta = \left|{\|\lambda x\|} -{\left\|\begin{bmatrix}
0 & {l}_{\lambda x}
\\ {r}_{\lambda x} & 0
\end{bmatrix}\right\|}\right| = 0.
\end{align*}
From this it follows that ${\|x\|} = {\left\|\begin{bmatrix}
0 & \overline{\lambda}{l}_{x}
\\ \lambda{r}_{x} & 0
\end{bmatrix}\right\|}$.

Conversely, if ${\|x\|} = {\left\|\begin{bmatrix}
0 & \overline{\lambda}{l}_{x}
\\ \lambda{r}_{x} & 0
\end{bmatrix}\right\|}$ for all $\lambda \in \mathbb{T}$, then
\begin{align*}
\frac{1}{2}\displaystyle{\sup_{\lambda \in \mathbb{T}}}{\left\|\begin{bmatrix}
0 & \overline{\lambda}{l}_{x}
\\ \lambda{r}_{x} & 0
\end{bmatrix}\right\|} = \frac{1}{2}{\|x\|},
\end{align*}
and so, by Theorem \ref{T.2.03}, $\Omega(x) = \frac{1}{2}{\|x\|}$.
\end{proof}
For every $a\in \mathscr{A}$ and $x\in\mathscr{V}$, by the inequalities (\ref{T.2.2.I.1}) and (\ref{T.2.2.I.3}), we have
\begin{align*}
\Omega(xa + xa^*) \leq {\|xa + xa^*\|} \leq 2 {\|a\|}{\|x\|}
\leq 4{\|a\|}\Omega(x),
\end{align*}
and hence
\begin{align}\label{T.2.5.I.1}
\Omega(xa + xa^*) \leq 4{\|a\|}\Omega(x).
\end{align}
In the following theorem, we improve the inequality (\ref{T.2.5.I.1}).
\begin{theorem}\label{T.2.5}
Let $\mathscr{V}$ be a Hilbert $\mathscr{A}$-module.
Let $a\in \mathscr{A}$ and $x\in\mathscr{V}$. Then
\begin{align*}
\Omega(xa + xa^*) \leq 2{\|a + a^*\|} \Omega(x).
\end{align*}
\end{theorem}
\begin{proof}
Let $\mathbb{L}(\mathscr{V})$ be the linking algebra of $\mathscr{V}$.
For every $b\in \mathscr{A}$ and $y\in\mathscr{V}$, we have
\begin{align*}
{r}_{xa}(b) = (xa)b = x(ab) = x(T_{a}(b)) = {r}_{x} T_{a}(b)
\end{align*}
and
\begin{align*}
{l}_{xa}(y) = \langle xa, y\rangle = a^*\langle x, y\rangle = a^*({l}_{x}(y)) = T_{a^*}{l}_{x}(y).
\end{align*}
Hence ${r}_{xa} = {r}_{x} T_{a}$ and ${l}_{xa} = T_{a^*}{l}_{x}$.
Now, let $\lambda \in \mathbb{T}$. Therefore,
\begin{align*}
{\left\|\begin{bmatrix}
0 & \overline{\lambda}{l}_{(xa + xa^*)}
\\ \lambda{r}_{(xa + xa^*)} & 0
\end{bmatrix}\right\|}
& = {\left\|\begin{bmatrix}
0 & \overline{\lambda}(T_{a^*}{l}_{x} + T_{a}{l}_{x})
\\ \lambda({r}_{x} T_{a} + {r}_{x} T_{a^*}) & 0
\end{bmatrix}\right\|}
\\& = {\left\|\begin{bmatrix}
0 & \overline{\lambda}T_{a + a^*}{l}_{x}
\\ \lambda{r}_{x}T_{a + a^*} & 0
\end{bmatrix}\right\|}
\\& = {\left\|\begin{bmatrix}
0 & \overline{\lambda}{l}_{x}
\\ \lambda{r}_{x} & 0
\end{bmatrix}\begin{bmatrix}
T_{a + a^*} & 0
\\ 0 & 0
\end{bmatrix}
+ \begin{bmatrix}
T_{a + a^*} & 0
\\ 0 & 0
\end{bmatrix}\begin{bmatrix}
0 & \overline{\lambda}{l}_{x}
\\ \lambda{r}_{x} & 0
\end{bmatrix}\right\|}
\\& \leq 2 {\left\|\begin{bmatrix}
T_{a + a^*} & 0
\\ 0 & 0
\end{bmatrix}\right\|}
{\left\|\begin{bmatrix}
0 & \overline{\lambda}{l}_{x}
\\ \lambda{r}_{x} & 0
\end{bmatrix}\right\|}
\\& \leq 4{\|a + a^*\|} \Omega(x),
\end{align*}
and so
\begin{align*}
\frac{1}{2}{\left\|\begin{bmatrix}
0 & \overline{\lambda}{l}_{(xa + xa^*)}
\\ \lambda{r}_{(xa + xa^*)} & 0
\end{bmatrix}\right\|}
\leq 2{\|a + a^*\|} \Omega(x).
\end{align*}
Taking the supremum over $\lambda \in \mathbb{T}$ in the above inequality, we deduce that
\begin{align*}
\Omega(xa + xa^*) \leq 2{\|a + a^*\|} \Omega(x).
\end{align*}
\end{proof}
As an immediate consequence of Theorem \ref{T.2.5}, we have the following result.
\begin{corollary}\label{C.2.6}
Let $\mathscr{V}$ be a Hilbert $\mathscr{A}$-module
and let $a\in \mathscr{A}$ and $x\in\mathscr{V}$. If $xa = xa^*$, then
\begin{align*}
\Omega(xa) \leq {\|a + a^*\|} \Omega(x).
\end{align*}
\end{corollary}
\begin{remark}\label{R.2.7}
Let $\mathscr{V}$ be a Hilbert $\mathscr{A}$-module
and let $a\in \mathscr{A}$ and $x\in\mathscr{V}$.
Replace $a$ by $ia$ in Theorem \ref{T.2.5}, to obtain
$\Omega(xa - xa^*) \leq 2{\|a - a^*\|} \Omega(x)$.
Thus
\begin{align*}
\Omega(xa \pm xa^*) \leq 2{\|a \pm a^*\|} \Omega(x).
\end{align*}
\end{remark}
In what follows, $r(a)$ stands for the spectral radius of
an arbitrary element $a$ in a $C^*$-algebra $\mathscr{A}$.
It is well known that for every $a\in \mathscr{A}$, we
have $r(a) \leq {\|a\|}$ and that equality holds in this inequality if $a$ is normal.
The following lemma gives us a spectral radius inequality for sums of elements in $C^*$-algebras.
\begin{lemma}\cite[Lemma 3.5]{Z}\label{L.2.8}
Let $\mathscr{A}$ be a $C^*$-algebra and let $a, b\in \mathscr{A}$. Then
\begin{align*}
r(a + b) \leq {\left\|\begin{bmatrix}
{\|a\|} & {\|ab\|}^{1/2}
\\ {\|ab\|}^{1/2} & {\|b\|}
\end{bmatrix}\right\|}.
\end{align*}
\end{lemma}
Now, we present a refinement of the triangle inequality
for the numerical radius in Hilbert $C^*$-modules.
We use some ideas of \cite[Theorem 3.4]{A.K.1}.
We refer the reader to \cite{N.W, E.M.P, Po, Raji.3} for more
information on the triangle inequality.
\begin{theorem}\label{T.2.9}
Let $\mathscr{V}$ be a Hilbert $\mathscr{A}$-module
and let $\mathbb{L}(\mathscr{V})$ be the linking algebra of $\mathscr{V}$.
Let $x, y\in \mathscr{V}$. Then
\begin{align*}
\Omega(x + y)
\leq {\left\|\begin{bmatrix}
\Omega(x) & \frac{1}{2}{\left\|\begin{bmatrix}
T_{\langle x, y\rangle} & 0
\\ 0 & \theta_{x, y}
\end{bmatrix}\right\|}^{1/2}
\\ \frac{1}{2}{\left\|\begin{bmatrix}
T_{\langle x, y\rangle} & 0
\\ 0 & \theta_{x, y}
\end{bmatrix}\right\|}^{1/2} & \Omega(y)
\end{bmatrix}\right\|}
\leq \Omega(x) + \Omega(y).
\end{align*}
\end{theorem}
\begin{proof}
Let $\lambda \in \mathbb{T}$.
Put $a = \begin{bmatrix}
0 & \overline{\lambda}{l}_{x}
\\ \lambda{r}_{x} & 0
\end{bmatrix}$
and $b = \begin{bmatrix}
0 & \overline{\lambda}{l}_{y}
\\ \lambda{r}_{y} & 0
\end{bmatrix}$.
Then
\begin{align*}
{\|a\|} \leq 2\Omega(x) \qquad \mbox{and} \qquad{\|b\|} \leq 2\Omega(y).
\end{align*}
Also, for every $c\in\mathscr{A}$ and $z\in\mathscr{V}$, we have
\begin{align*}
{l}_{x}{r}_{y}(c) = {l}_{x}(yc) = \langle x, yc\rangle = \langle x, y\rangle c = T_{\langle x, y\rangle}(c)
\end{align*}
and
\begin{align*}
{r}_{x}{l}_{y}(z) = {r}_{x}(\langle y, z\rangle) = x\langle y, z\rangle = \theta_{x, y}(z).
\end{align*}
Thus ${l}_{x}{r}_{y} = T_{\langle x, y\rangle}$ and ${r}_{x}{l}_{y} = \theta_{x, y}$.
Therefore,
$ab = \begin{bmatrix}
T_{\langle x, y\rangle} & 0
\\ 0 & \theta_{x, y}
\end{bmatrix}$ and hence,
\begin{align}\label{T.2.9.I.1}
{\left\|\begin{bmatrix}
T_{\langle x, y\rangle} & 0
\\ 0 & \theta_{x, y}
\end{bmatrix}\right\|} = {\|ab\|}
\leq {\|a\|}{\|b\|}
\leq 4\Omega(x)\Omega(y).
\end{align}
Since $\begin{bmatrix}
0 & \overline{\lambda}{l}_{(x + y)}
\\ \lambda{r}_{(x + y)} & 0
\end{bmatrix}$ is a self adjoint element of $C^*$-algebra $\mathbb{L}(\mathscr{V})$,
we have
\begin{align*}
{\left\|\begin{bmatrix}
0 & \overline{\lambda}{l}_{(x + y)}
\\ \lambda{r}_{(x + y)} & 0
\end{bmatrix}\right\|}
= r\left(\begin{bmatrix}
0 & \overline{\lambda}{l}_{(x + y)}
\\ \lambda{r}_{(x + y)} & 0
\end{bmatrix}\right).
\end{align*}
Therefore, by Lemma \ref{L.2.8}, we obtain
\begin{align*}
{\left\|\begin{bmatrix}
0 & \overline{\lambda}{l}_{(x + y)}
\\ \lambda{r}_{(x + y)} & 0
\end{bmatrix}\right\|}
&= r\left(\begin{bmatrix}
0 & \overline{\lambda}{l}_{(x + y)}
\\ \lambda{r}_{(x + y)} & 0
\end{bmatrix}\right)
\\& =  r(a + b)
\\& \leq {\left\|\begin{bmatrix}
{\|a\|} & \|ab\|^{1/2}
\\ \|ab\|^{1/2} & {\|b\|}
\end{bmatrix}\right\|}.
\end{align*}
So, by the norm monotonicity of matrices with nonnegative entries
(see, e.g., \cite[p. 491]{H.J}), we get
\begin{align*}
{\left\|\begin{bmatrix}
0 & \overline{\lambda}{l}_{(x + y)}
\\ \lambda{r}_{(x + y)} & 0
\end{bmatrix}\right\|}
&\leq {\left\|\begin{bmatrix}
\displaystyle{\sup_{\lambda \in \mathbb{T}}}{\|a\|} & \displaystyle{\sup_{\lambda \in \mathbb{T}}}{\|ab\|}^{1/2}
\\ \displaystyle{\sup_{\lambda \in \mathbb{T}}}{\|ab\|}^{1/2} & \displaystyle{\sup_{\lambda \in \mathbb{T}}}{\|b\|}
\end{bmatrix}\right\|}
\\& = {\left\|\begin{bmatrix}
2\Omega(x) & {\left\|\begin{bmatrix}
T_{\langle x, y\rangle} & 0
\\ 0 & \theta_{x, y}
\end{bmatrix}\right\|}^{1/2}
\\ {\left\|\begin{bmatrix}
T_{\langle x, y\rangle} & 0
\\ 0 & \theta_{x, y}
\end{bmatrix}\right\|}^{1/2} & 2\Omega(y)
\end{bmatrix}\right\|}.
\end{align*}
Therefore, for every $\lambda \in \mathbb{T}$ we have
\begin{align*}
\frac{1}{2}{\left\|\begin{bmatrix}
0 & \overline{\lambda}{l}_{(x + y)}
\\ \lambda{r}_{(x + y)} & 0
\end{bmatrix}\right\|}
\leq {\left\|\begin{bmatrix}
\Omega(x) & \frac{1}{2}{\left\|\begin{bmatrix}
T_{\langle x, y\rangle} & 0
\\ 0 & \theta_{x, y}
\end{bmatrix}\right\|}^{1/2}
\\ \frac{1}{2}{\left\|\begin{bmatrix}
T_{\langle x, y\rangle} & 0
\\ 0 & \theta_{x, y}
\end{bmatrix}\right\|}^{1/2} & \Omega(y)
\end{bmatrix}\right\|},
\end{align*}
whence
\begin{align}\label{T.2.9.I.2}
\Omega(x + y)
\leq {\left\|\begin{bmatrix}
\Omega(x) & \frac{1}{2}{\left\|\begin{bmatrix}
T_{\langle x, y\rangle} & 0
\\ 0 & \theta_{x, y}
\end{bmatrix}\right\|}^{1/2}
\\ \frac{1}{2}{\left\|\begin{bmatrix}
T_{\langle x, y\rangle} & 0
\\ 0 & \theta_{x, y}
\end{bmatrix}\right\|}^{1/2} & \Omega(y)
\end{bmatrix}\right\|}.
\end{align}
On the other hand, by (\ref{T.2.9.I.1}), we have
\begin{align}\label{T.2.9.I.3}
&{\left\|\begin{bmatrix}
\Omega(x) & \frac{1}{2}{\left\|\begin{bmatrix}
T_{\langle x, y\rangle} & 0
\\ 0 & \theta_{x, y}
\end{bmatrix}\right\|}^{1/2}
\\ \frac{1}{2}{\left\|\begin{bmatrix}
T_{\langle x, y\rangle} & 0
\\ 0 & \theta_{x, y}
\end{bmatrix}\right\|}^{1/2} & \Omega(y)
\end{bmatrix}\right\|}\nonumber
\\& \qquad = \frac{1}{2}\left(\Omega(x) + \Omega(y) + \sqrt{(\Omega(x) - \Omega(y))^2 + {\left\|\begin{bmatrix}
T_{\langle x, y\rangle} & 0
\\ 0 & \theta_{x, y}
\end{bmatrix}\right\|}}\right)
\\& \qquad \leq \frac{1}{2}\left(\Omega(x) + \Omega(y) + \sqrt{(\Omega(x) - \Omega(y))^2 + 4\Omega(x)\Omega(y)}\right)
= \Omega(x) + \Omega(y)\nonumber.
\end{align}
Thus
\begin{align*}
{\left\|\begin{bmatrix}
\Omega(x) & \frac{1}{2}{\left\|\begin{bmatrix}
T_{\langle x, y\rangle} & 0
\\ 0 & \theta_{x, y}
\end{bmatrix}\right\|}^{1/2}
\\ \frac{1}{2}{\left\|\begin{bmatrix}
T_{\langle x, y\rangle} & 0
\\ 0 & \theta_{x, y}
\end{bmatrix}\right\|}^{1/2} & \Omega(y)
\end{bmatrix}\right\|}
\leq \Omega(x) + \Omega(y),
\end{align*}
and the proof is completed.
\end{proof}
As a consequence of Theorem \ref{T.2.9}, we have the following result.
\begin{corollary}\label{C.2.10}
Let $\mathscr{V}$ be a Hilbert $\mathscr{A}$-module, and
$x, y\in \mathscr{V}$. If $\Omega(x + y) = \Omega(x) + \Omega(y)$, then
\begin{align*}
\Omega(x)\Omega(y) = \frac{1}{4}{\left\|\begin{bmatrix}
T_{\langle x, y\rangle} & 0
\\ 0 & \theta_{x, y}
\end{bmatrix}\right\|}.
\end{align*}
In particular,
$\Omega(x) = \frac{1}{2}{\left\|\begin{bmatrix}
T_{\langle x, x\rangle} & 0
\\ 0 & \theta_{x, x}
\end{bmatrix}\right\|}^{1/2}$.
\end{corollary}
The following lemma must be known to specialists. For the sake of completeness we include the proof.
\begin{lemma}\label{L.2.11}
Let $\mathscr{V}$ be a Hilbert $\mathscr{A}$-module, and $x, y\in \mathscr{V}$.
Then
\begin{align*}
\left\|\theta_{x, y}\right\| = \left\|{\langle x, x\rangle}^{1/2}{\langle y, y\rangle}^{1/2}\right\|.
\end{align*}
\end{lemma}
\begin{proof}
We may assume that $x, y \neq 0$ otherwise the identity trivially holds.
We have
\begin{align*}
\left\|\theta_{x, y}\left(\frac{y{\langle x, x\rangle}^{1/2}}{\left\|y{\langle x, x\rangle}^{1/2}\right\|}\right)\right\|^2
&= \frac{\left\|x\langle y, y\rangle{\langle x, x\rangle}^{1/2}\right\|^2}{\left\|y{\langle x, x\rangle}^{1/2}\right\|^2}
\\& = \frac{\left\|{\langle x, x\rangle}^{1/2}\langle y, y\rangle\langle x, x\rangle\langle y, y\rangle{\langle x, x\rangle}^{1/2}\right\|}
{\left\|{\langle x, x\rangle}^{1/2}\langle y, y\rangle{\langle x, x\rangle}^{1/2}\right\|}
\\& = \left\|{\langle x, x\rangle}^{1/2}\langle y, y\rangle{\langle x, x\rangle}^{1/2}\right\|
= \left\|{\langle x, x\rangle}^{1/2}{\langle y, y\rangle}^{1/2}\right\|^2,
\end{align*}
and so
\begin{align*}
\left\|\theta_{x, y}\left(\frac{y{\langle x, x\rangle}^{1/2}}{\left\|y{\langle x, x\rangle}^{1/2}\right\|}\right)\right\|
= \left\|{\langle x, x\rangle}^{1/2}{\langle y, y\rangle}^{1/2}\right\|.
\end{align*}
Hence
\begin{align}\label{I.1.L.2.11}
\left\|\theta_{x, y}\right\| \geq \left\|{\langle x, x\rangle}^{1/2}{\langle y, y\rangle}^{1/2}\right\|.
\end{align}
On the other hand, let $z\in\mathscr{V}$ with $\|z\|=1$. By \eqref{C.S.I} we have
$\langle y, z\rangle\langle z, y\rangle \leq \langle y, y\rangle$
and hence by Theorem 2.2.5(2) of \cite{Mu} it follows that
\begin{align*}
{\langle x, x\rangle}^{1/2}\langle y, z\rangle\langle z, y\rangle{\langle x, x\rangle}^{1/2}
\leq {\langle x, x\rangle}^{1/2}\langle y, y\rangle{\langle x, x\rangle}^{1/2}.
\end{align*}
So, \cite[Theorem 2.2.5(3)]{Mu} implies
\begin{align}\label{I.2.L.2.11}
\left\|{\langle x, x\rangle}^{1/2}\langle y, z\rangle\langle z, y\rangle{\langle x, x\rangle}^{1/2}\right\|
\leq \left\|{\langle x, x\rangle}^{1/2}\langle y, y\rangle{\langle x, x\rangle}^{1/2}\right\|.
\end{align}
Therefore,
\begin{align*}
\left\|\theta_{x, y}(z)\right\| &= \left\|x\langle y, z\rangle\right\|
\\&= \left\|\langle z, y\rangle\langle x, x\rangle\langle y, z\rangle\right\|^{1/2}
\\& = \left\|{\langle x, x\rangle}^{1/2}\langle y, z\rangle\langle z, y\rangle{\langle x, x\rangle}^{1/2}\right\|^{1/2}
\\& \stackrel{\eqref{I.2.L.2.11}}{\leq} \left\|{\langle x, x\rangle}^{1/2}\langle y, y\rangle{\langle x, x\rangle}^{1/2}\right\|^{1/2}
= \left\|{\langle x, x\rangle}^{1/2}{\langle y, y\rangle}^{1/2}\right\|,
\end{align*}
whence
\begin{align}\label{I.3.L.2.11}
\left\|\theta_{x, y}\right\| \leq \left\|{\langle x, x\rangle}^{1/2}{\langle y, y\rangle}^{1/2}\right\|.
\end{align}
Utilizing \eqref{I.1.L.2.11} and \eqref{I.3.L.2.11}, we conclude that
$\left\|\theta_{x, y}\right\| = \left\|{\langle x, x\rangle}^{1/2}{\langle y, y\rangle}^{1/2}\right\|$.
\end{proof}
We close this paper with the following result.
\begin{corollary}\label{C.2.12}
Let $\mathscr{V}$ be a Hilbert $\mathscr{A}$-module, and
$x, y\in \mathscr{V}$. If $\langle x, y\rangle = 0$, then
\begingroup\makeatletter\def\f@size{10}\check@mathfonts
\begin{align*}
\Omega(x + y)
\leq \frac{1}{2}\left(\Omega(x) + \Omega(y) + \sqrt{(\Omega(x) - \Omega(y))^2 + \left\|{\langle x, x\rangle}^{1/2}{\langle y, y\rangle}^{1/2}\right\|}\right)
\leq \Omega(x) + \Omega(y).
\end{align*}
\endgroup
\end{corollary}
\begin{proof}
Since $\langle x, y\rangle = 0$, we have $T_{\langle x, y\rangle} = 0$.
Hence from \eqref{T.2.9.I.2}, \eqref{T.2.9.I.3} and Lemma \ref{L.2.11}, we deduce the desired result.
\end{proof}
\bibliographystyle{amsplain}

\end{document}